\documentclass[reqno]{amsart}
\usepackage{amssymb,amsmath,amsfonts,amscd,amsthm,pb-diagram}
\usepackage{amsbsy,bm}
\usepackage[usenames,dvipsnames]{color}
\usepackage[normalem]{ulem}

\newtheorem{theorem}{Theorem}[section]

\newtheorem{lemma}{Lemma}[section]
\newtheorem{proposition}{Proposition}[section]
\newtheorem{corollary}{Corollary}[section]
\newtheorem{remark}{Remark}[section]

\newtheorem{conjecture}{Conjecture}[section]
\newtheorem{problem}{Problem}[section]

\begin{document}
\title[Hypersurfaces with semi-parallel cubic form]
{Conformally flat affine hypersurfaces with semi-parallel cubic form}

\author{Huiyang Xu}
\address{
School of Mathematics and Statistics, Henan University of Science
and Technology, Luoyang 471023, People's Republic of China.}
\email{xuhuiyang@haust.edu.cn}

\author{Cece Li}
\address{
School of Mathematics and Statistics, Henan University of Science
and Technology, Luoyang 471023, People's Republic of China.}
\email{ceceli@haust.edu.cn}


\thanks{Conflict of interest.
The authors have no conflict of interest.}

\thanks{This work was supported by NNSF of China
(Nos. 12101194, 11401173).}

\date{}


\begin{abstract}
In this paper, we study locally strongly convex affine hypersurfaces
with vanishing Weyl curvature tensor and semi-parallel cubic form relative
to the Levi-Civita connection of affine metric. 
As the main result, we classify such hypersurfaces being not of flat affine metric. In particular,
$2, 3$-dimensional locally strongly convex affine hypersurfaces with
semi-parallel cubic form are completely determined.
\end{abstract}

\maketitle
\numberwithin{equation}{section}

\section{Introduction}\label{sect:1}
The classical equiaffine differential geometry is mainly concerned with geometric
properties and invariants of hypersurfaces in affine space, that are invariant under
unimodular affine transformations. Let $\mathbb{R}^{n+1}$ be the $(n+1)$-dimensional
real unimodular affine space. On any non-degenerate hypersurface immersion of
$\mathbb{R}^{n+1}$, it is well known how to induce the affine connection $\nabla$,
the affine shape operator $S$ whose eigenvalues are called affine principal curvatures,
and the affine metric $h$. The classical Pick-Berwald theorem states that
the cubic form $C:=\nabla h$ vanishes, if and only if the hypersurface is a
non-degenerate hyperquadric. In that sense the cubic form plays the role as
the second fundamental form for submanifolds of real space forms.

In the past decades, similar to the Pick-Berwald theorem, geometric conditions on
the cubic form have been used to classify natural classes of affine hypersurfaces
by many geometers, see e.g. \cite{BD,BNS,BV,DV1,G,H,HLLV1,HLLV2,O,W}.
Among them, one of the most interesting developments may be the
classification of locally strongly convex affine hypersurfaces with $\hat\nabla C=0$,
where $\hat\nabla$ is the Levi-Civita connection of affine metric.
In this subject, F. Dillen, L. Vrancken, et al. obtain the classifications for lower
dimensions in \cite{DV,DVY,HLSV,MN}, and finally Z. Hu, H. Li and L. Vrancken
complete the classification for all dimensions as follows.

\begin{theorem}[cf. \cite{HLV}]\label{thm:1.1}
Let $M$ be an $n$-dimensional ($n\geq2$) locally strongly convex affine hypersurface in
$\mathbb{R}^{n+1}$ with $\hat\nabla C=0$. Then $M$ is either a
hyperquadric (i.e., $C=0$) or a hyperbolic affine hypersphere with $C\not=0$;
in the latter case either
\begin{enumerate}

\item[(i)] $M$ is obtained as the Calabi product of a lower
dimensional hyperbolic affine hypersphere with parallel cubic form and a
point, or

\item[(ii)] $M$ is obtained as the Calabi product of two lower
dimensional hyperbolic affine hyperspheres with parallel cubic form, or

\item[(iii)] $n=\tfrac12m(m+1)-1,\ m\ge3$, $(M,h)$ is isometric
to ${\bf SL}(m,\mathbb{R})/{\bf SO}(m)$, and $M$ is affinely
equivalent to the standard embedding ${\bf SL}(m,\mathbb{R})/{\bf
SO}(m)\hookrightarrow\mathbb{R}^{n+1}$, or

\item[(iv)] $n=m^2-1,\ m\ge3$, $(M,h)$ is isometric to
${\bf SL}(m,\mathbb{C})/{\bf SU}(m)$, and $M$ is affinely equivalent
to the standard embedding ${\bf SL}(m,\mathbb{C})/{\bf
SU}(m)\hookrightarrow\mathbb{R}^{n+1}$, or

\item[(v)] $n=2m^2-m-1,\ m\ge3$, $(M,h)$ is isometric to
${\bf SU}^*\big(2m\big)/{\bf Sp}(m)$, and $M$ is affinely equivalent
to the standard embedding ${\bf SU}^*\big(2m\big)/{\bf
Sp}(m)\hookrightarrow\mathbb{R}^{n+1}$, or

\item[(vi)] $n=26$, $(M,h)$ is isometric to
${\bf E}_{6(-26)}/{\bf F}_4$, and $M$ is affinely equivalent to the
standard embedding ${\bf E}_{6(-26)}/{\bf F}_4\hookrightarrow
\mathbb{R}^{27}$.
\end{enumerate}
\end{theorem}

As that did in \cite{HLV, HX}, we say that an affine hypersurface has semi-parallel
(resp. parallel) cubic form relative to the Levi-Civita connection of affine metric
if $\hat{R}\cdot C=0$ (resp. $\hat\nabla C=0$), where $\hat{R}$ is the
curvature tensor of affine metric, and the tensor $\hat{R}\cdot C$ is defined by
\begin{equation}\label{eqn:1.1}
\begin{array}{lll}
\hat{R}(X,Y)\cdot C=\hat\nabla_X\hat\nabla_YC-\hat\nabla_Y\hat\nabla_XC-\hat\nabla_{[X,Y]}C
\end{array}
\end{equation}
for tangent vector fields $X,Y$. Obviously, the parallelism of cubic form implies
its semi-parallelism, the converse is not true, we refer to
Remark \ref{rem:3.1} for the counter-examples.

As a generalization of Theorem \ref{thm:1.1}, naturally one may ask for the following problem:
\begin{problem}\label{prob:1.1}
Classify all the $n$-dimensional locally strongly convex affine hypersurfaces
with $\hat{R}\cdot C=0$.
\end{problem}

Related to Problem \ref{prob:1.1}, Z. Hu and C. Xing \cite{HX} show that
such surface is either a quadric, or a flat affine surface. Recently, the present
authors and C. Xing \cite{LXX} give an answer to Problem \ref{prob:1.1}
in two cases for such hypersurfaces with at most one affine principal curvature
of multiplicity one: being not affine hyperspheres, being affine hyperspheres
with constant scalar curvature. Based on the latter, they pose the following

\begin{conjecture}\label{conj:1.1}
For any locally strongly convex affine hypersphere of dimension $n$,
the semi-parallelism and parallelism of cubic form are equivalent,
i.e., $\hat{R}\cdot C=0$ iff $\hat\nabla C=0$.
\end{conjecture}

In this paper, we continue to study locally strongly convex affine hypersurfaces
with $\hat{R}\cdot C=0$, and pay our attention to the case: Weyl curvature tensor
vanishes identically. First, by investigating such affine hypersphere
we can confirm Conjecture \ref{conj:1.1} in this case.

\begin{theorem}\label{thm:1.2}
Let $M^n$, $n\geq3$, be a locally strongly convex affine hypersphere in $\mathbb{R}^{n+1}$
with $\hat{R}\cdot C=0$ and vanishing Weyl curvature tensor.
Then $M^n$ is affinely equivalent to either a hyperquadric, or
the flat and hyperbolic affine hypersphere:
\begin{align}
x_1x_2\cdots x_{n+1}=1,\label{eqn:1.2}
\end{align}
or the hyperbolic affine hypersphere with quasi-Einstein affine metric:
\begin{align}
(x_{n}^2-x_1^2-\cdots-x_{n-1}^2)^nx_{n+1}^2=1,\label{eqn:1.3}
\end{align}
where $(x_1,\ldots,x_{n+1})$ are the standard coordinates of $\mathbb{R}^{n+1}$.
\end{theorem}

From Theorem \ref{thm:1.2} and the fact that Weyl curvature tensor
vanishes automatically for $n=3$ we obtain immediately:
\begin{corollary}\label{cor:1.1}
Let $M^3$ be a locally strongly convex affine hypersphere in $\mathbb{R}^{4}$
with $\hat{R}\cdot C=0$.
Then $M^3$ is affinely equivalent to either a hyperquadric, or one of
the two hyperbolic affine hyperspheres:
\begin{align*}
&x_1x_2x_3x_4=1,\\[2mm]
&(x_3^2-x_1^2-x_2^2)^3x_4^2=1,
\end{align*}
where $(x_1,\ldots,x_{4})$ are the standard coordinates of $\mathbb{R}^{4}$.
\end{corollary}
\begin{remark}\label{rem:1.1}
Conjecture \ref{conj:1.1} is true for $n=2,3,4$ by the following three facts:
\begin{enumerate}
\item[(1)] By Theorem 1.1 and Corollary 2.1 in \cite{CH}, Theorem \ref{thm:1.2} confirms
Conjecture \ref{conj:1.1} if either $n=3$, or $M^n$ being conformally flat for $n\geq4$;

\item[(2)] It has been shown in \cite{LXX} that Conjecture \ref{conj:1.1} is true if either $n=2$,
or $M^n$ being constant scalar curvature for $n\geq3$.

\item[(3)] It was shown in Lemma 3.1 of \cite{DVY} that
if $\hat{R}\cdot C=0$ and $n=4$, then the Pick invariant $J$ can only take four
constant values at any point, which implies that locally $J$ and thus the scalar curvature
is constant. This together with (2) confirms Conjecture \ref{conj:1.1} for $n=4$.
\end{enumerate}
\end{remark}

Second, by removing the restriction of affine hyperspheres we give an
answer to Problem \ref{prob:1.1} under the assumption
that Weyl curvature tensor vanishes identically.

\begin{theorem}\label{thm:1.3}
Let $M^n$, $n\geq3$, be a locally strongly convex affine hypersurface in
$\mathbb{R}^{n+1}$ with $\hat{R}\cdot C=0$ and vanishing Weyl curvature tensor.
Denote by $m$ the number of the distinct eigenvalues of its Schouten tensor.
Then $m\leq2$, and one of the two cases occurs:
\begin{enumerate}
\item[(i)] $m=1$, $M^n$ is either a hyperquadric, or a flat affine hypersurface;

\item[(ii)] $m=2$, $M^n$ is affinely equivalent to either \eqref{eqn:1.3},
or one of the six quasi-umbilical affine hypersurfaces with quasi-Einstein metric
and nonzero scalar curvature, explicitly described in Theorem \ref{thm:4.1}.
\end{enumerate}
\end{theorem}
\begin{remark}\label{rem:1.2}
$m=1$ in case (i) means that affine metric is constant sectional curvature.
Related to this, we prove in Proposition \ref{prop:3.1}
that any non-degenerate affine hypersurface satisfies $\hat{R}\cdot C=0$
and affine metric being constant sectional curvature if and only if
it is either a hyperquadric or a flat affine hypersurface.
\end{remark}
\begin{remark}\label{rem:1.3}
The case (i) is partially classified in Remark \ref{rem:4.1}, its complete classification is
still complicated and not solved. In fact, Anti\'c-Li-Vrancken-Wang \cite{ALVW} recently
classify the locally strongly convex affine hypersurfaces with constant sectional curvature
if the hypersurface admits at most one affine principal curvature of multiplicity one.
Even for affine surfaces with flat affine metric, the classification problem
is still open up to now.
\end{remark}
\begin{remark}\label{rem:1.4}
The examples in case (ii) are the generalized Calabi compositions of
a hyperquadric and a point in some special forms \cite{ADSV}. The construction
method of such examples initially originates from E. Calabi \cite{C},
and now has been extended and characterized by M. Anti\'c, F. Dillen,
L. Vrancken, Z. Hu, H. Li, et al. in \cite{ADSV,AHLV,DV94,HLV1}.
\end{remark}

This paper is organized as follows. In Section \ref{sect:2}, we briefly review
the local theory of equiaffine hypersurfaces, some notions and results
of conformally flat manifolds and warped product manifolds.
In Section \ref{sect:3}, we study the properties of the hypersurfaces
involving the eigenvalues and eigenvalue distributions of Schouten tensor,
the difference tensor and the affine principal curvatures, and present the
proof of Theorem \ref{thm:1.2}. Based on these properties and known results,
in Section \ref{sect:4} we discuss all the possibilities of the immersion and
complete the proof of Theorem \ref{thm:1.3}.

\section{Preliminaries}\label{sect:2}

In this section, we briefly review the local theory of equiaffine hypersurfaces.
For more details, we refer to the monographs \cite{LSZH,NS}.

Let $\mathbb{R}^{n+1}$ denote the standard $(n+1)$-dimensional real unimodular
affine space that is endowed with its usual flat connection $D$ and a parallel volume
form $\omega$, given by the determinant. Let $F: M^{n} \rightarrow \mathbb{R}^{n+1}$
be an oriented non-degenerate hypersurface immersion. On such a hypersurface, up to a sign
there exists a unique transversal vector field $\xi$, called the {\it affine normal}.
A non-degenerate hypersurface equipped with the affine normal is called an
{\it (equi)affine hypersurface}, or a {\it Blaschke hypersurface}.
Denote always by $X,Y,Z,U$ the tangent vector fields on $M^{n}$.
By the affine normal we can write
\begin{align}
\qquad &D_XF_*Y=F_*\nabla_XY+h(X,Y)\xi,  \label{eqn:2.1}\hspace{2cm}&(\textrm{Gauss formula})\\
&D_X\xi=-F_*SX,  \label{eqn:2.2}&(\textrm{Weingarten formula})
\end{align}
which induce on $M^n$ the {\it affine connection} $\nabla$, a semi-Riemannian
metric $h$, called the {\it affine metric}, the {\it affine shape operator} $S$
and the {\it cubic form} $C:=\nabla h$. An affine hypersurface is called
{\it locally strongly convex} if $h$ is definite, we always choose $\xi$,
up to a sign, such that $h$ is positive definite. We call a locally strongly convex affine
hypersurface {\it quasi-umbilical} (resp. {\it quasi-Einstein}) if it admits exactly two
distinct eigenvalues of $S$ (resp. Ricci tensor of $h$), one of which is simple.

Let $\hat{\nabla}$ be the Levi-Civita connection of the affine metric $h$.
The difference tensor $K$ is defined by
\begin{equation}\label{eqn:2.3}
K(X,Y):=\nabla_XY-\hat{\nabla}_XY.
\end{equation}
We also write $K_XY$
and $K_X=\nabla_X-\hat{\nabla}_X$.
Since both $\nabla$ and $\hat{\nabla}$ have zero torsion,
$K$ is symmetric in $X$ and $Y$. It is
related to the totally symmetric cubic form $C$ by
\begin{equation}\label{eqn:2.4}
C(X, Y, Z)=-2h(K(X, Y), Z),
\end{equation}
which implies that the operator $K_X$ is symmetric relative to $h$.
Moreover, $K$ satisfies the {\it apolarity condition}, namely,
${\rm tr}\, K_X=0$ for all $X$.

The curvature tensor $\hat{R}$ of affine metric $h$, $S$ and $K$ are related
by the following Gauss and Codazzi equations:
\begin{equation}\label{eqn:2.5}
\begin{array}{lll}
&\hat{R}(X,Y)Z=\tfrac{1}{2}[h(Y,Z)SX-h(X,Z)SY+h(SY,Z)X-h(SX,Z)Y]\\[2mm]
&\qquad\qquad\qquad -[K_X,K_Y]Z,
\end{array}
\end{equation}
\begin{equation}\label{eqn:2.6}
\begin{array}{lll}
&(\hat{\nabla}_XK)(Y,Z)-(\hat{\nabla}_YK)(X,Z)\\[2mm]
&\qquad\qquad=\frac12[h(Y,Z)SX-h(X,Z)SY-h(SY,Z)X+h(SX,Z)Y],
\end{array}
\end{equation}
\begin{equation}\label{eqn:2.7}
(\hat{\nabla}_XS)Y-(\hat{\nabla}_YS)X=K(SX,Y)-K(SY,X),
\qquad\qquad\qquad\quad\
\end{equation}
where, by definitions, $[K_X,K_Y]Z=K_XK_YZ-K_YK_XZ$, and
\begin{equation*}
\begin{array}{lll}
&\hat{R}(X,Y)Z=\hat{\nabla}_X\hat{\nabla}_YZ
-\hat{\nabla}_{Y}\hat{\nabla}_XZ-\hat{\nabla}_{[X,Y]}Z,\\[2mm]
&(\hat{\nabla}_XK)(Y,Z)=\hat{\nabla}_X(K(Y,Z))
-K(\hat{\nabla}_XY,Z)-K(Y,\hat{\nabla}_XZ),\\[2mm]
&(\hat{\nabla}_XS)Y=\hat{\nabla}_X(SY)-S\hat{\nabla}_XY.
\end{array}
\end{equation*}
Contracting Gauss equation \eqref{eqn:2.5} twice
we have
\begin{equation}\label{eqn:2.8}
\chi=H+J,
\end{equation}
where $J=\tfrac1{n(n-1)}h(K,K)$, $H=\tfrac{1}{n}{\rm tr}\,S$, $\chi=\tfrac{r}{n(n-1)}$
and $r$ are the {\it Pick invariant}, {\it affine mean curvature}, {\it normalized scalar curvature}
and scalar curvature of $h$, respectively. Recall the following Ricci identity:
\begin{equation*}
\begin{array}{lll}
(\hat{R}(X, Y)\cdot K)(Z, U)=\hat{R}(X, Y)K(Z, U)-K(\hat{R}(X,Y)Z, U)-K(Z, \hat{R}(X,Y)U).
\end{array}
\end{equation*}

$M^n$ is called an affine hypersphere if $S=H\, id$.
Then it follows from \eqref{eqn:2.7} that $H$ is constant if $n\geq2$. $M^n$
is said to be a proper (resp. improper) affine hypersphere if $H$ is nonzero (resp. zero).
Moreover, a locally strongly convex affine hypersphere is called parabolic, elliptic
or hyperbolic according to $H=0$, $H>0$ or $H<0$, respectively.
For affine hyperspheres, the Gauss and Codazzi equations reduce to
\begin{equation}\label{eqn:2.9}
\begin{aligned}
\hat{R}(X,Y)Z=H[h(Y,Z)X-h(X,Z)Y]-[K_{X},K_{Y}]Z,
\end{aligned}
\end{equation}
\begin{equation}\label{eqn:2.10}
\begin{aligned}
(\hat{\nabla}_XK)(Y,Z)=(\hat{\nabla}_YK)(X,Z).
\end{aligned}
\end{equation}

We collect the following two results for later use.

\begin{theorem}[cf. Theorem 1.1 and Corollary 2.1 of \cite{CH}]\label{thm:2.1}
Let $M^n$, $n\geq3$, be a locally strongly convex affine hypersphere in $\mathbb{R}^{n+1}$
with constant scalar curvature $r$. Then $Jr\leq0$, and the traceless Ricci tensor
$\tilde{Ric}$ satisfies
$$
\|\tilde{Ric}\|_h^2\leq-\tfrac{(n+1)(n-2)}{n+2}Jr,
$$
where $\|\cdot\|_h$ denotes the tensorial norm with respect to $h$.
This equality sign holds identically, if and only if $M^n$ has parallel cubic form and
vanishing Weyl curvature tensor, if and only if one of three cases occurs:
\begin{enumerate}
\item[(i)]
$J=0$, $M^n$ is affinely equivalent to a hyperquadric;

\item[(ii)]
$J\neq0$, $r=0$, $M^n$ is affinely equivalent to \eqref{eqn:1.2};

\item[(iii)]
$J\neq0$, $r<0$, $M^n$ is affinely equivalent to \eqref{eqn:1.3}.
\end{enumerate}
\end{theorem}

\begin{theorem}[cf. Theorem 1 of \cite{ADSV}]\label{thm:2.2}
Let $M^{m+1}$, $m\geq2$, be a locally strongly convex
affine hypersurface of the affine space $\mathbb{R}^{m+2}$ such that its
tangent bundle is an orthogonal sum, with respect to the affine metric $h$,
of two distributions: a one-dimensional distribution $\mathcal{D}_1$ spanned
by a unit vector field $T$ and an $m$-dimensional distribution
$\mathcal{D}_2$, such that
\begin{align*}
& K(T,T)=\lambda_1 T,\quad K(T, X)=\lambda_2 X, \\
& ST=\mu_1 T,\quad SX=\mu_2 X, \quad\forall\ X\in
\mathcal{D}_2.
\end{align*}
Then either $M^{m+1}$ is an affine hypersphere such that $K_T=0$
or is affinely congruent to one of the following immersions:
\begin{enumerate}
\item[(1)] $f(t,x_1,\dots,x_m)=(\gamma_1(t),\gamma_2(t)g_2(x_1,\dots,x_m))$,
for $\gamma_1,\gamma_2$ such that
$$
\epsilon\gamma_1'\gamma_2(\gamma_1'\gamma_2''- \gamma_1''\gamma_2')<0;
$$
\item[(2)] $f(t,x_1,\dots,x_m)=\gamma_1(t)C(x_1,\dots,x_m)+\gamma_2(t)e_{m+1}$,
for $\gamma_1, \gamma_2$ such that
$$
{\rm sgn}\Big(\gamma_1'\gamma_2''-\gamma_1''\gamma_2'\Big)
={\rm sgn}(\gamma_1'\gamma_1)\neq0;
$$

\item[(3)] $f(t,x_1,\dots,x_m)=C(x_1,\dots,x_m)+\gamma_2(t)e_{m+1}+\gamma_1(t)e_{m+2}$,
for $\gamma_1, \gamma_2$ such that
$$
{\rm sgn}(\gamma_1'\gamma_2''-\gamma_1''\gamma_2')={\rm sgn} (\gamma_1')\neq0.
$$
\end{enumerate}
Here $g_2:\mathbb R^m\rightarrow \mathbb R^{m+1}$ is a proper affine hypersphere
centered at the origin with affine mean curvature $\epsilon$, and $C:\mathbb R^{m}
\rightarrow \mathbb R^{m+2}$ is an improper affine hypersphere, given by
$C(x_1,\dots,x_m)=(x_1,\dots,x_m, p(x_1,\dots,x_m),1)$,
with the affine normal $e_{m+1}$.
\end{theorem}

Finally, we review some notions and results of conformally flat manifolds and
warped product manifolds. The Schouten tensor of $(1,1)$-type,
on a Riemannian manifold $(M, h)$ of dimension $n\geq3$,
is a self-adjoint operator relative to $h$, defined by
\begin{equation}\label{eqn:2.11}
P=\tfrac{Q}{n-2}-\tfrac{r}{2(n-1)(n-2)}I,
\end{equation}
where $Q$, $I$ and $r$ are the Ricci operator, identity operator and scalar curvature, respectively.
Note that $(M,h)$ is conformally flat, which means that a neighborhood
of each point can be conformally immersed into the Euclidean space $\mathbb{R}^n$.
For $n\geq4$, $(M,h)$ is conformally fat if and only if the Weyl curvature tensor
defined below vanishes.
\begin{equation*}
W(X,Y)Z=\hat{R}(X,Y)Z-[h(Y,Z)PX-h(X,Z)PY+h(PY,Z)X-h(PX,Z)Y].
\end{equation*}
It is well known that for $n=3$, $W=0$ identically, and $(M,h)$ is conformally flat if and only if
the Schouten tensor is a Codazzi tensor. If $W=0$, the Riemannian curvature tensor
can be rewritten by
\begin{equation}\label{eqn:2.12}
\hat{R}(X,Y)Z=h(Y,Z)PX-h(X,Z)PY+h(PY,Z)X-h(PX,Z)Y.
\end{equation}
$(M,h)$ of dimension $n\geq3$ is of constant sectional curvature if and only if $W=0$
and its metric is Einstein, where Einstein metric means that the eigenvalues of
the operator $P$ or $Q$ are single.

For Riemannian manifolds $(B,g_B)$, $(M_1,g_1)$ and a positive
function $f:B\rightarrow \mathbb{R}$, the product manifold $M:=B\times M_1$,
equipped with the metric $h=g_B\oplus f^2g_1$, is called a warped product manifold
with the warped function $f$, denoted by $B\times_{f}M_1$.
If $E$ and $U, V$ are independent vector fields of $B$ and $M_1$, respectively,
then (cf. (2) of \cite{BGV}) the sectional curvatures of $M$ satisfy
\begin{equation}\label{eqn:2.13}
\begin{array}{lll}
K^M_{EU}=-\tfrac{H_{f}(E,E)}{fh(E,E)},\
K^M_{UV}=\tfrac{1}{f^2}K^{M_1}_{UV}
-\tfrac{h(\hat{\nabla}f,\hat{\nabla}f)}{f^2},
\end{array}
\end{equation}
where $\hat{\nabla}$ is the Levi-Civita connection of $(M, h)$,
$K^M$ and $K^{M_1}$ are the sectional curvatures of
$M$ and $M_1$, respectively, $\hat{\nabla}f$ is the gradient
of the function $f$ and $H_{f}(X,Y)=h(\hat{\nabla}_X\hat{\nabla}f,Y)$
is the Hessian of $f$. Related to the conformally flat manifold, we recall

\begin{theorem}[cf. Theorem 3.7 of \cite{BGV}]\label{thm:2.3}
Let $M=B\times_{f}M_1$ be a warped product with dim~$B=1$
and dim~$M_1\geq2$. Then $M$ is conformally flat if and only if,
up to a reparametrization of $B$, $(M_1,g_1)$ is a space of
constant curvature and $f$ is any positive function.
\end{theorem}

\section{Conformally flat affine hypersurfaces with $\hat{R}\cdot C=0$}\label{sect:3}
From this section on, when we say that an affine hypersurface has {\it semi-parallel cubic form},
it always means that $\hat{R}\cdot C=0$, equivalently $\hat{R}\cdot K=0$. Then,
by the Ricci identity of $K$ we have
\begin{equation}\label{eqn:3.1}
\hat{R}(X, Y)K(Z, U)=K(\hat{R}(X,Y)Z, U)+K(Z, \hat{R}(X,Y)U).
\end{equation}
In fact, by \eqref{eqn:2.4} and the Ricci identities of $C$ and $K$,
the equivalence above follows from the following formula:
\begin{equation}\label{eqn:3.2}
\begin{array}{lll}
&(\hat{R}\cdot C)(U,V,X,Y,Z)=(\hat{R}(U,V)\cdot C)(X,Y,Z)\\
&=-C(X,Y,\hat{R}(U,V)Z)-C(X,\hat{R}(U,V)Y,Z)-C(\hat{R}(U,V)X,Y,Z)\\
&=2[h(K_XY,\hat{R}(U,V)Z)+h(K_X\hat{R}(U,V)Y,Z)+h(K_Y\hat{R}(U,V)X,Z)]\\
&=-2h(\hat{R}(U,V)K_XY-K_X\hat{R}(U,V)Y-K_Y\hat{R}(U,V)X,Z)\\
&=-2h((\hat{R}(U,V)\cdot K)(X,Y),Z).
\end{array}
\end{equation}

First, under the assumption of the affine metric being constant sectional curvature,
we extend the result of Theorem 6.2 in \cite{HX},
stating that a locally strongly convex affine surface has
semi-parallel cubic form if and only if it is either a quadric or a flat affine surface,
from the surface to the non-degenerate affine hypersurface as follows.

\begin{proposition}\label{prop:3.1}
Let $M^n$, $n\geq2$, be a non-degenerate affine hypersurface of $\mathbb{R}^{n+1}$.
Then $M^n$ has semi-parallel cubic form and affine metric being constant
sectional curvature, if and only if $M^n$ is either a hyperquadric,
or a flat affine hypersurface.
\end{proposition}
\begin{proof}
The `` if part " follows from \eqref{eqn:3.1} and the Pick-Berwald theorem.

Now, we prove the `` only if part ". Assume that the affine metric
is of constant sectional curvature $c$, then
\begin{equation}\label{eqn:3.3}
\hat{R}(X,Y)Z=c(h(Y,Z)X-h(X,Z)Y).
\end{equation}
Let us choose an orthonormal basis $\{e_1, \ldots, e_n\}$ such
that $h(e_i,e_j)=\epsilon_i\delta_{ij}$ and $\epsilon_i=\pm1$. Then,
from \eqref{eqn:3.1} we get
\begin{equation*}
h(e_i,\hat{R}(e_i,e_j)K(e_i,e_i))=2h(e_i,K(\hat{R}(e_i,e_j)e_i,e_i)),
\end{equation*}
which together with \eqref{eqn:3.3} implies that, for $i\neq j$,
\begin{equation*}
\begin{aligned}
0=h(e_i,\hat{R}(e_i,e_j)K_{e_i}e_i-2K_{e_i}\hat{R}(e_i,e_j)e_i)
=3c\epsilon_ih(K_{e_i}e_i,e_j).
\end{aligned}
\end{equation*}
Therefore, either $c=0$, or $c\neq0$ and
$K_{e_i}e_i=c_ie_i$ for all $i$. In the latter case, the apolarity
condition shows that $c_j={\rm tr}\,K_{e_j}=0$ for each $j$,
and thus $K_{e_i}e_i=0$ for all $i$. As $K$ is symmetric, we have $K=0$ identically.
By Pick-Berwald theorem the conclusion follows.
\end{proof}
\begin{remark}\label{rem:3.1}
To see the examples
whose cubic form are semi-parallel but not parallel,
we refer to Remark 6.2 in \cite{HX} for such flat surfaces,
and Theorem 4.1 in \cite{ALVW} for the flat quasi-umbilical affine hypersurfaces.
\end{remark}

From now on, let $F: M^n\rightarrow\mathbb{R}^{n+1}$, $n\geq3$, be a locally strongly
convex affine hypersurface with $\hat{R}\cdot C=0$ and vanishing Weyl curvature tensor.
Then we have \eqref{eqn:2.12}.
Denote by $\{e_1,\ldots, e_n\}$ the local orthonormal frame of $M^n$, where
$e_i$ are the eigenvector fields of the Schouten tensor $P$ with corresponding
eigenvalues $\nu_i$, $i=1,\ldots, n$.
Then, we see from \eqref{eqn:2.12} that
\begin{equation}\label{eqn:3.4}
\hat{R}(e_i,e_j)Z=(\nu_i+\nu_j)(h(e_j,Z)e_i-h(e_i,Z)e_j)
\end{equation}
for any tangent vector $Z$.
Taking $X=e_i, Y=e_j$, $W=e_k, Z=e_\ell$ in \eqref{eqn:3.1}
we have
\begin{equation}\label{eqn:3.5}
\begin{aligned}
&(\nu_i+\nu_j)(h(K_{e_j}e_k,e_\ell)e_{i}
-h(K_{e_i}e_k,e_\ell)e_{j})\\[2mm]
&\quad =(\nu_i+\nu_j)(\delta_{jk}K_{e_i}e_\ell
-\delta_{ik}K_{e_j}e_\ell+\delta_{j\ell}K_{e_i}e_k
-\delta_{i\ell}K_{e_j}e_k).
\end{aligned}
\end{equation}

For $k=\ell=i\neq j$ in \eqref{eqn:3.5} we obtain that
\begin{equation}\label{eqn:3.6}
\begin{aligned}
(\nu_i+\nu_j)(2K_{e_i}e_j +h(K_{e_i}e_i,e_j)e_{i}
-h(K_{e_i}e_i,e_i)e_{j})=0.
\end{aligned}
\end{equation}
Taking the inner product of \eqref{eqn:3.6} with $e_i$ we deduce that
\begin{equation}\label{eqn:3.7}
(\nu_i+\nu_j)h(K_{e_i}e_i,e_j)=0,\ \forall \ i\neq j.
\end{equation}
Interchanging the role of $e_i$ and $e_j$, similarly we have
\begin{equation}\label{eqn:3.8}
\begin{aligned}
(\nu_j+\nu_i)h(K_{e_j}e_j,e_i)=0,\ \forall \ i\neq j.
\end{aligned}
\end{equation}
Taking the inner product of \eqref{eqn:3.6} with $e_j$, by \eqref{eqn:3.8} we see that
\begin{equation}\label{eqn:3.9}
(\nu_i+\nu_j)h(K_{e_i}e_i,e_i)=0, \ \forall \ i\neq j.
\end{equation}
Together with \eqref{eqn:3.7} we obtain from \eqref{eqn:3.6} that
\begin{equation}\label{eqn:3.10}
(\nu_i+\nu_j)K_{e_i}e_j=0, \ \forall \ i\neq j.
\end{equation}
On the other hand, set $k=i, \ell=j$ in \eqref{eqn:3.5}, by \eqref{eqn:3.7}
and \eqref{eqn:3.8} we deduce that
\begin{equation}\label{eqn:3.11}
(\nu_i+\nu_j)(K_{e_i}e_i-K_{e_j}e_j)=0, \ \forall \ i\neq j.
\end{equation}

For more information, we denote by $\nu_1,\ldots,\nu_m$ the $m$ distinct eigenvalues for $P$
of multiplicity $(n_1,n_2,\ldots,n_m)$, respectively. Let $\mathfrak{D}(\nu_i)$
be the eigenvalue distribution of eigenvalue $\nu_i$ for $i=1,\ldots, m$.
Then, we can prove the following two lemmas.
\begin{lemma}\label{lem:3.1}
It holds that
\begin{enumerate}
\item[(i)]
If $\nu_i\neq0$ and $n_i\geq2$, then
\begin{equation*}
h(K(u,v),w)=0,\ \forall\ u,v,w\in\mathfrak{D}(\nu_i).
\end{equation*}

\item[(ii)]
If $\nu_i^2\neq\nu_j^2$, then
\begin{equation*}
K(u,v)=0, \quad \forall \ u,v\in \mathfrak{D}(\nu_i)\oplus\mathfrak{D}(\nu_j).
\end{equation*}
\end{enumerate}
\end{lemma}
\begin{proof}
If $\nu_i\neq0$ and $n_i\geq2$, then, by taking $\nu_j=\nu_i$ and $e_i=u$ in \eqref{eqn:3.9} we have
$h(K(u,u),u)=0$ for any unit vector $u\in\mathfrak{D}(\nu_i)$. Then, the conclusion (i)
follows from the symmetric property of $K$.

If $\nu_i^2\neq\nu_j^2$, there hold
\begin{equation}\label{eqn:3.12}
(\nu_i+\nu_j)(\nu_i-\nu_j)\neq0,\ m\geq2.
\end{equation}
Then, by \eqref{eqn:3.9} we have $h(K(u,u),u)=0$ for any vector
$u\in\mathfrak{D}(\nu_i)$, which together with the symmetric property
of $K$ implies that $K(u,u)\notin \mathfrak{D}(\nu_i)$. Furthermore,
it follows from \eqref{eqn:3.10} and \eqref{eqn:3.11} that
\begin{equation}\label{eqn:3.13}
K(u,v)=0, \quad K(u,u)=K(v,v)\notin \mathfrak{D}(\nu_i)\oplus\mathfrak{D}(\nu_j)
\end{equation}
for any unit vectors $u\in\mathfrak{D}(\nu_i)$, $v\in\mathfrak{D}(\nu_j)$.
If $m=2$, \eqref{eqn:3.13} immediately implies that $K(u,u)=K(v,v)=0$.
If $m\geq3$, for arbitrary $\nu_k$, different from $\nu_i$ and $\nu_j$,
by \eqref{eqn:3.12} either $\nu_k+\nu_i\neq0$
or $\nu_k+\nu_j\neq0$ holds. In either case, by \eqref{eqn:3.7}
and the the arbitrariness of $\nu_k$ we see from
\eqref{eqn:3.13} that $K(u,u)=K(v,v)=0$. Therefore,
by the symmetric property of $K$ we have (ii).
\end{proof}
\begin{lemma}\label{lem:3.2}
Let $M^n$ be a locally strongly convex affine hypersurface in $\mathbb{R}^{n+1}$
with $\hat{R}\cdot C=0$ and vanishing Weyl curvature tensor.
Then,
\begin{enumerate}
\item[(i)]
The number $m$ of distinct eigenvalues of the Schouten tensor $P$
is at most $2$;

\item[(ii)]
If $m=2$, i.e., namely $\nu_1\neq\nu_2$,
then one of their multiplicity must be $1$, and
\begin{equation}\label{eqn:3.14}
\begin{aligned}
\nu_2=-\nu_1\neq0.
\end{aligned}
\end{equation}
\end{enumerate}
\end{lemma}
\begin{proof}
First, for $m\geq2$, we claim that there must exist two distinct eigenvalues
$\nu_i$ and $\nu_j$ of $P$ such that $\nu_i+\nu_j=0$. Otherwise, we have
$\nu_i^2\neq\nu_j^2$ for any $\nu_i\neq \nu_j$,
which together with Lemma \ref{lem:3.1} (ii) implies that $K=0$ identically.
Then, the Pick-Berwald theorem implies that the hypersurface
is a hyperquadric, and thus of constant sectional curvature.
This means that $m=1$, a contradiction to $m\geq2$.

By the claim above, for $m\geq2$ we denote by $\nu_1, \nu_2$ the two
distinct eigenvalues of $P$ such that $\nu_1+\nu_2=0$,
and thus $\nu_1\nu_2\neq0$.

Next, we prove $m\leq2$. Otherwise, if $m\geq3$, then, for arbitrary $\nu_i$,
different from $\nu_1$ and $\nu_2$,
we have $(\nu_i+\nu_1)(\nu_i+\nu_2)\neq0$, and thus
\begin{equation*}
\nu_i^2\neq\nu_1^2,\quad \nu_i^2\neq\nu_2^2.
\end{equation*}
It follows from Lemma \ref{lem:3.1} (ii) that
\begin{equation}\label{eqn:3.15}
\begin{aligned}
&K(w,u)=K(w,v)=0, \\
&K(u,u)=K(v,v)=K(w,w)=0
\end{aligned}
\end{equation}
for any vectors $u\in\mathfrak{D}(\nu_1)$, $v\in\mathfrak{D}(\nu_2)$,
$w\in\mathfrak{D}(\nu_i)$. Therefore, the arbitrariness of $\nu_i$ implies
that $K(u,v)=0$. Together with \eqref{eqn:3.15}  and the symmetric property of $K$
we have $K=0$ identically. As before, it follows from the Pick-Berwald theorem
that $m=1$, a contradiction to $m\geq3$. The conclusion (i) follows.

Finally, for $m=2$, let $\nu_1,\nu_2$ be the two distinct
eigenvalues of $P$ with multiplicity $(n_1, n_2)$, respectively.
From the analysis as above we have \eqref{eqn:3.14}.
It is sufficient to prove that either $n_1=1$ or $n_2=1$ holds.
On the contrary, assume that $n_1\geq2$ and $n_2\geq2$.
Denote by $v^i_1,\ldots, v^i_{n_i}$ the orthonormal eigenvector fields
of $P$, which span $\mathfrak{D}(\nu_i)$ for $i=1,2$.
Then, as $\nu_i\neq0$ and $n_i\geq2$, by Lemma \ref{lem:3.1} (i)
and \eqref{eqn:3.11} we have
\begin{equation}\label{eqn:3.16}
\begin{aligned}
&h(K(v^i_p,v^i_q),v^i_\ell)=0,\\
&K(v^i_p,v^i_p)=K(v^i_q,v^i_q),\ \forall\ p,q, \ell
\end{aligned}
\end{equation}
for $i=1,2$. It follows from the apolarity condition that
$$
0=\sum_{\ell=1}^{n_1}h(K_{u}v^1_\ell,v^1_\ell)+
\sum_{q=1}^{n_2}h(K_{u}v^2_q,v^2_q)
=\sum_{q=1}^{n_2}h(K_{u}v^2_q,v^2_q)
=n_2h(K_{u}v^2_1,v^2_1)
$$
for any $u\in \mathfrak{D}(\nu_1)$. Therefore, $h(K_{u}v,v)=0$
for any $u\in \mathfrak{D}(\nu_1)$, $v\in \mathfrak{D}(\nu_2)$.
Similarly, we have $h(K_{v}u,u)=0$.
Then, by the symmetric property of $K$ we have $K(u,v)=0$.
Together with \eqref{eqn:3.16} we further obtain that $K=0$ identically.
As before, this is a contradiction to $m=2$. The conclusion (ii) follows.
\end{proof}

Next, based on Lemma \ref{lem:3.2}, we pay our attention to the case $m=2$.
We always denote by $\nu_1,\nu_2$ the two distinct eigenvalues of $P$ with
multiplicity $1, n-1$, respectively. Let $T$ be the unit eigenvector field of
the eigenvalue $\nu_1$, and $\{X_1,\ldots, X_{n-1}\}$ be any orthonormal
frame of $\mathfrak{D}(\nu_2)$.
By \eqref{eqn:3.14}, \eqref{eqn:3.10}, \eqref{eqn:3.11} and
Lemma \ref{lem:3.1} (i) we see that
 \begin{equation}\label{eqn:3.17}
\begin{aligned}
&K(X_i, X_i)=K(X_j, X_j), \\
&K(X_i, X_j)=0,\ \forall\ i\neq j,\\
&h(K_{X_i}X_j,X_k)=0,\ \forall\ i, j, k.
\end{aligned}
\end{equation}
It follows from the apolarity condition that
\begin{equation*}
\begin{aligned}
&h(K(T,T), X_i)=-\sum_{j=1}^{n-1}h(K(X_j,X_j),X_i)=0,\\
&h(K(T,T), T)=-\sum_{j=1}^{n-1}h(K(X_j,X_j),T)=-(n-1)h(K(X_i,X_i),T),\ \forall\ i.
\end{aligned}
\end{equation*}
Summing above, as $m=2$ and thus $K\neq0$, we can assume that
 \begin{equation}\label{eqn:3.18}
\begin{aligned}
&K_TT=\lambda_1T,\ K_TX_i=\lambda_2X_i, \ K_{X_i}X_j=\lambda_2\delta_{ij}T,\\
&PT=\nu_1T,\quad PX_i=\nu_2X_i,\ \ i, j=1,\ldots,n-1,\\
&(n-1)\lambda_2=-\lambda_1\neq0, \ \nu_2=-\nu_1\neq0.
\end{aligned}
\end{equation}
Combining the last formula with \eqref{eqn:2.11} and the fact that tr\,$Q=r$ we obtain that
\begin{equation}\label{eqn:3.19}
\begin{aligned}
&QT=0,\ QX_i=\tfrac{r}{n-1}X_{i}, \ i=1,\ldots, n-1,\\
&\nu_2=-\nu_1=\tfrac{r}{2(n-1)(n-2)}\neq0,
\end{aligned}
\end{equation}
which imply that the affine metric is quasi-Einstein with nonzero scalar curvature.

Now, we are ready to prove the following

\begin{lemma}\label{lem:3.3}
If $m=2$, then the number $\sigma$ of the distinct affine principal curvatures
is at most 2. Moreover, there exists a local orthonormal frame,
still denoted by $\{T, X_1,\ldots, X_{n-1}\}$, such that
there hold \eqref{eqn:3.18}, \eqref{eqn:3.19} and
\begin{equation}\label{eqn:3.20}
\begin{aligned}
&ST=\mu_1T,\ SX_i=\mu_{2}X_{i}, \ i=1,\ldots, n-1,\\
&\mu_{2}=2\nu_2+\lambda_2^2,\ \mu_1+\mu_{2}=-2n\lambda_2^2.
\end{aligned}
\end{equation}
\end{lemma}
\begin{proof}
From the analysis as above we still have the freedom to rechoose
the orthonormal frame of $\mathfrak{D}(\nu_2)$, such that
\eqref{eqn:3.18} and \eqref{eqn:3.19} hold.
Therefore, we can reselect the orthonormal frame of $\mathfrak{D}(\nu_2)$
if necessary, still denoted by $\{X_1,\ldots, X_{n-1}\}$, such that
\begin{equation}\label{eqn:3.21}
ST=\mu_1T+\tau_1X_1,\ SX_1=\tau_1T+\mu_2X_1+\tau_2X_2.
\end{equation}
Then, it follows from \eqref{eqn:3.18}, \eqref{eqn:3.4} and
the Gauss equation \eqref{eqn:2.5} that
 \begin{equation*}
\begin{aligned}
&0=h(\hat{R}(X_1,T)T,X_2)=\tfrac12\tau_2,\\
&0=h(\hat{R}(X_1,X_2)X_2,T)=\tfrac12\tau_1,
\end{aligned}
\end{equation*}
which together with \eqref{eqn:3.21} imply that both $\mathfrak{D}(\nu_1)$ and
$\mathfrak{D}(\nu_2)$ are the invariant subspace of $S$. Therefore,
we can rechoose the local orthonormal frame on $\mathfrak{D}(\nu_2)$,
still denoted by $\{X_1,\ldots, X_{n-1}\}$, such that
\begin{equation*}
ST=\mu_1T,\ SX_{i-1}=\mu_{i}X_{i-1}, \ i=2,\ldots, n.
\end{equation*}
Then, we see from \eqref{eqn:3.18}, \eqref{eqn:3.4} and
the Gauss equation \eqref{eqn:2.5} that
 \begin{equation*}
\begin{aligned}
&0=\nu_1+\nu_2=h(\hat{R}(X_{i-1},T)T,X_{i-1})
=\tfrac12(\mu_1+\mu_{i})+n\lambda_2^2,\ \forall\ i\geq2,\\
&2\nu_2=h(\hat{R}(X_{i-1},X_{j-1})X_{j-1},X_{i-1})
=\tfrac12(\mu_{i}+\mu_{j})-\lambda_2^2,\ \forall\ i\neq j\geq2,
\end{aligned}
\end{equation*}
which imply that
\begin{equation*}
\begin{aligned}
\mu_{i}=-\mu_1-2n\lambda_2^2=2\nu_2+\lambda_2^2,\ i=2,\ldots,n.
\end{aligned}
\end{equation*}
Therefore, we have $\sigma\leq2$ and \eqref{eqn:3.20}.
\end{proof}

Finally, we conclude this section by proving Theorem \ref{thm:1.2}.

\begin{proof}[\bf Completion of Theorem \ref{thm:1.2}'s Proof]
Let $F: M^n\rightarrow\mathbb{R}^{n+1}$, $n\geq3$, be a locally strongly
convex affine hypersphere with $\hat{R}\cdot C=0$ and vanishing Weyl curvature tensor.
Then, we see from Lemma \ref{lem:3.2} (i) that the number $m$ of distinct eigenvalues of
the Schouten tensor $P$ is at most $2$.

If $m=1$, by $W=0$ we see from \eqref{eqn:2.12} that $M^n$ is of constant sectional curvature.
Then, it follows from Proposition \ref{prop:3.1} and Main theorem of \cite{VLS}
that $M^n$ is affinely equivalent to either a hyperquadric, or
the flat and hyperbolic affine hypersphere \eqref{eqn:1.2}.

If $m=2$, based on Lemma \ref{lem:3.3}, by the fact that $H=\mu_1=\mu_2$ is constant,
we see from \eqref{eqn:3.20}, \eqref{eqn:3.19} and \eqref{eqn:2.8} that
\begin{equation}\label{eqn:3.22}
\begin{aligned}
&H=-n\lambda_2^2<0,\\
&r=n(n-1)\chi=\tfrac{(n^2-1)(n-2)}{n}H<0,\\
&J=-\tfrac{n+2}{n^2}H>0,
\end{aligned}
\end{equation}
and thus $r$ is a negative constant. Define the traceless Ricci tensor
$\tilde{Ric}=\hat{Ric}-\tfrac{r}{n}h$, we deduce from \eqref{eqn:3.19}
and \eqref{eqn:3.22} that all the components of $\tilde{Ric}$
vanish except
\begin{equation*}
\tilde{Ric}(T,T)=-\tfrac{r}{n},\
\tilde{Ric}(X_1,X_1)=\cdots=\tilde{Ric}(X_{n-1},X_{n-1})=\tfrac{r}{n(n-1)}.
\end{equation*}
Combining with \eqref{eqn:3.22} we have
\begin{equation}\label{eqn:3.23}
\|\tilde{Ric}\|_h^2=\tfrac{r^2}{n(n-1)}=-\tfrac{(n+1)(n-2)}{n+2}Jr,
\end{equation}
where $\|\cdot\|_h$ denotes the tensorial norm with respect to $h$.

In summary, $M^n$ is a hyperbolic affine hypersphere
with constant negative scalar curvature and $J\neq0$, which further
satisfies the formula \eqref{eqn:3.23}.
Then, it follows from Theorem \ref{thm:2.1} that
$M^n$ is affinely equivalent to the hyperbolic affine hypersphere \eqref{eqn:1.3}.
\end{proof}

\section{Proof of Theorem \ref{thm:1.3}}\label{sect:4}
Let $M^n$, $n\geq3$, be a locally strongly convex affine hypersurface in
$\mathbb{R}^{n+1}$ with $\hat{R}\cdot C=0$ and vanishing
Weyl curvature tensor. Denote by $m$ (resp. $\sigma$)
the number of the distinct eigenvalues of its Schouten tensor (resp. affine shape operator).
Then, by Lemma \ref{lem:3.2} (i) we have $m\leq2$.

If $m=1$, by $W=0$ we see from \eqref{eqn:2.12} that $M^n$ is of
constant sectional curvature. Then, it follows from Proposition \ref{prop:3.1}
that $M^n$ is either a hyperquadric, or a flat affine hypersurface.

\begin{remark}\label{rem:4.1}
If $M^n$ is a flat affine hypersurface with $C\neq0$, it follows from
Theorems 1.1 and 1.2 of \cite{ALVW} and Theorem \ref{thm:1.2} that
$M^n$ is affinely equivalent to either \eqref{eqn:1.2} if $\sigma=1$, or one
of the three flat quasi-umbilical affine hypersurface (1), (2), (7) explicitly
described in Theorem 4.1 of \cite{ALVW} if $\sigma=2$. However,
for $\sigma\geq3$, $M^n$ is a flat affine hypersurface with at least two
affine principal curvatures of multiplicity one, whose classification
is still complicated and not involved up to now.
\end{remark}

If $m=2$, by Lemma \ref{lem:3.3} we have $\sigma\leq2$.
For $\sigma=1$, i.e., $M^n$ is an affine hypersphere,
Theorem \ref{thm:1.2} shows that $M^n$ is affinely equivalent to \eqref{eqn:1.3}.

For the last case: $m=\sigma=2$, we see from Lemma \ref{lem:3.3} that
$M^n$ is a quasi-umbilical affine hypersurface with quasi-Einstein affine metric
with nonzero scalar curvature. Moreover, for the local orthonormal frame
$\{T, X_1,\ldots, X_{n-1}\}$ there hold
 \begin{equation}\label{eqn:4.1}
\begin{aligned}
&K_TT=\lambda_1T,\ K_TX_i=\lambda_2X_i, \ K_{X_i}X_j=\lambda_2\delta_{ij}T,\\
&ST=\mu_1T,\quad SX_i=\mu_2X_i,\ \ i, j=1,\ldots,n-1,
\end{aligned}
\end{equation}
where
 \begin{equation}\label{eqn:4.2}
\begin{aligned}
\mu_{2}-\lambda_2^2=\tfrac{r}{(n-1)(n-2)}\neq0,\qquad &\mu_1+\mu_{2}=-2n\lambda_2^2,\\
\lambda_1=-(n-1)\lambda_2\neq0,\qquad &\mu_1-\mu_{2}\neq0.
\end{aligned}
\end{equation}
By \eqref{eqn:3.18} and \eqref{eqn:3.19} we further obtain from \eqref{eqn:2.12} that
 \begin{equation}\label{eqn:4.3}
\begin{aligned}
&\hat{R}(X_i, X_j)X_k=\tfrac{r}{(n-1)(n-2)}(\delta_{jk}X_i-\delta_{ik}X_j),\\
&\hat{R}(X_i, T)T=\hat{R}(T, X_i)X_i=0,\  \forall\ 1\leq i, j, k\leq n-1.
\end{aligned}
\end{equation}

\begin{remark}\label{rem:4.2}
For $\sigma=m=2$, $M^n$ is affinely equivalent to one of the three classes
of immersions in Theorem \ref{thm:2.2}. In fact, it follows from
\eqref{eqn:4.1} that $M^n$ satisfies the conditions of Theorem \ref{thm:2.2}.
In the proof of Theorem \ref{thm:2.2} in \cite{ADSV}, it was shown in Lemma 3 that
if $\lambda_2=0$, then $K_T=0$ and $M^n$ is an affine hypersphere.
In our situation, by $M^n$ being not an affine hypersphere
we can exclude this possibility in Theorem \ref{thm:2.2}.
\end{remark}

Next, we show more information about the three classes of warped product
immersions in Theorem \ref{thm:2.2}.
Denote by $\mathfrak{D}(\mu_2)=$span$\{X_1,\ldots, X_{n-1}\}$
the eigenvalue distribution of $S$ corresponding to $\mu_2$.
In the proof of Theorem \ref{thm:2.2} in \cite{ADSV}, by \eqref{eqn:4.2}
it was shown that
\begin{equation}\label{eqn:4.4}
\begin{aligned}
&\hat{\nabla}_{T}T=0,\ \hat{\nabla}_{X}T=-\alpha X,\ T(\alpha)=\alpha^2,\\
&X(\alpha)=X(\mu_1)=X(\mu_2)=X(\lambda_2)=0,\ \forall \ X\in\mathfrak{D}(\mu_2),\\
&T(\lambda_2)=(n+1)\lambda_2\alpha+\tfrac12(\mu_1-\mu_2),\\
&T(\mu_2)=(\mu_2-\mu_1)(\alpha-\lambda_2),
\end{aligned}
\end{equation}
and $M^n$ is locally a warped product $\mathbb{R} \times _{f}M_2$, where
the warping function $f$ is determined by $\alpha=-T(\ln f)$,
$\mathbb{R}$ and $M_2$ are, respectively, integral manifolds of the
distributions span$\{T\}$ and $\mathfrak{D}(\mu_2)$,
and $M_2$ is an affine hypersphere.
The same proof implies that the projection of the difference tensor
on $\mathfrak{D}(\mu_2)$ is the difference tensor
of the components $M_2$, given by
\begin{equation}\label{eqn:4.5}
L^2(X,Y)=K(X,Y)-\lambda_2h(X,Y)T,\
\forall\ X, Y\in\mathfrak{D}(\mu_2).
\end{equation}
Then, it follows from \eqref{eqn:4.1} that $L^2=0$.
Hence $M_2$ is a hyperquadric with constant sectional curvature $c$.

By the warped product structure, if no other stated,
we always take the local coordinates $\{t,x_1,\cdots,x_{n-1}\}$ on $M^n$ such that
$\tfrac{\partial}{\partial t}=T$, span$\{\tfrac{\partial}{\partial
x_1},\cdots,\tfrac{\partial}{\partial x_{n-1}}\}=\mathfrak{D}(\mu_2)$.
Then all functions $\mu_i, \lambda_i, \alpha, r$ and $f$ depend only on $t$.
Denote by $\partial_t()=(\cdot)'$, we have $\alpha=-f'/f$.
It follows from \eqref{eqn:2.13} and \eqref{eqn:4.3} that
\begin{equation}\label{eqn:4.6}
c=f'^2+\tfrac{r}{(n-1)(n-2)}f^2.
\end{equation}
Furthermore, we can solve from the equations above for $f$ and $\alpha$ to get that,
up to a translation and a direction of the parametric $t$,
\begin{equation}\label{eqn:4.7}
f=1,\ \alpha=0;\  {\rm or} \ f=t,\ \alpha=-\tfrac{1}{t},
\end{equation}
where locally we take $t>0$. Together with \eqref{eqn:4.2} and \eqref{eqn:4.4} we
can check that $(\mu_2-\lambda_2^2)'=0$ if $f=1$,
and $((\mu_2-\lambda_2^2)t^2)'=0$ if $f=t$.
It follows from \eqref{eqn:4.2} and \eqref{eqn:4.6} that
\begin{equation}\label{eqn:4.8}
c= \left\{
\begin{array}{lll}
\tfrac{r}{(n-1)(n-2)}=\mu_2-\lambda_2^2,\ &{\rm if} \ f=1,\\[2mm]
1+\tfrac{r}{(n-1)(n-2)}t^2 =1+(\mu_2-\lambda_2^2)t^2,  &{\rm if} \ f=t,
\end{array}
\right.
\end{equation}
where the scalar curvature $r$ is nonzero in either case.

The Pick-Berwald theorem states that the cubic form or difference tensor vanishes,
if and only if the hypersurface is a non-degenerate hyperquadric.
For locally strongly convex case, for later use we recall from \cite{LSZH} on
pages 104-105 the following three hyperquadrics revised of dimension $n-1$:
\begin{enumerate}
\item[(1)] The ellipsoid, described in $\mathbb{R}^{n}$ by
$$
x_1^2+\cdots+x_n^2=c^{-(n+1)/n}, \ c>0.
$$
It has constant sectional curvature $c$ and the affine normal $-cF_1$,
where in local coordinates we denote this immersion by $F_1(x_1,\ldots,x_{n-1})$.

\item[(2)] The hyperboloid, described in $\mathbb{R}^{n}$ by
$$
x_n=(x_1^2+\cdots+x_{n-1}^2+(-c)^{-(n+1)/n})^{1/2},\ c<0.
$$
It has constant sectional curvature $c$ and the affine normal $-cF_2$,
where in local coordinates we denote this immersion by $F_2(x_1,\ldots,x_{n-1})$.

\item[(3)] The elliptic paraboloid, described in $\mathbb{R}^{n}$ by
$$
x_n=\tfrac12(x_1^2+\cdots+x_{n-1}^2).
$$
The affine metric is flat. This is the only parabolic affine hypersphere
with constant sectional curvature.
\end{enumerate}

Finally, armed with the above preparations, by Remark \ref{rem:4.2}
and the computations in \cite{ALVW} for the three classes of immersions in
Theorem \ref{thm:2.2}, we can prove the following theorem, which together
with previous cases completes the proof of Theorem \ref{thm:1.3}.

\begin{theorem}\label{thm:4.1}
Let $M^n$, $n\geq3$, be a locally strongly convex affine hypersurface in
$\mathbb{R}^{n+1}$ with $\hat{R}\cdot C=0$ and vanishing Weyl curvature tensor.
Assume that both the numbers of the distinct eigenvalues of its Schouten tensor
and affine shape operator are $2$. Then $M^n$ is a quasi-umbilical affine hypersurface
with quasi-Einstein metric and nonzero scalar curvature $r$. Moreover, $(M^n,h)$ is locally
isometric to the warped product $\mathbb{R_+} \times _{f}M_2$, where
the warped function $f(t)=1$ or $t$, $M_2$ is of constant sectional curvature $c$,
and $M^n$ is affinely equivalent to one of the following six hypersurfaces:
\begin{enumerate}
\item[(1)]  The immersion $(\gamma_1(t), \gamma_2(t)F_1(x_1,\ldots,x_{n-1}))$, where
\begin{enumerate}
\item[(a)] $f(t)=1$, and $c=\tfrac{r}{(n-1)(n-2)}$ is positive constant,

\item[(b)] $\gamma_1$ and $\gamma_2$ are determined by
\begin{equation*}
\begin{aligned}
&\gamma_1'=\gamma_2^{-n},\\
&\gamma_2=(c_1\cos(\sqrt{(n+1)c}t)+c_2\sin(\sqrt{(n+1)c}t))^{\tfrac{1}{n+1}},
\end{aligned}
\end{equation*}
where $c_1, c_2$ are constants such that $\gamma_2>0$.
\end{enumerate}

\item[(2)] The immersion $(\gamma_1(t), \gamma_2(t)F_2(x_1,\ldots,x_{n-1}))$, where
\begin{enumerate}
\item[(a)] $f(t)=1$, and $c=\tfrac{r}{(n-1)(n-2)}$ is negative constant,

\item[(b)] $\gamma_1$ and $\gamma_2$ are determined by
\begin{equation*}
\begin{aligned}
&\gamma_1'=\gamma_2^{-n},\\
&\gamma_2=(c_1e^{\sqrt{-(n+1)c}t}+c_2e^{-\sqrt{-(n+1)c}t})^{\tfrac{1}{n+1}},
\end{aligned}
\end{equation*}
where $c_1, c_2$ are constants such that $\gamma_2>0$.
\end{enumerate}

\item[(3)] The immersion $(\gamma_1(t), \gamma_2(t)F_1(x_1,\ldots,x_{n-1}))$, where
\begin{enumerate}
\item[(a)] $f(t)=t$, and $c=1+\tfrac{r}{(n-1)(n-2)}t^2\neq1$ is positive constant,

\item[(b)] $\gamma_1$ and $\gamma_2$ are determined by
\begin{equation*}
\gamma_1'=t^{n+1}\gamma_2^{-n},\
\gamma_2=k(t)^{1/(n+1)},
\end{equation*}
where $k(t)$ is a positive solution to the linear differential equation
\begin{equation*}
t^2k''(t)-(n+1)tk'(t)+(n+1)ck(t)=0.
\end{equation*}
\end{enumerate}

\item[(4)] The immersion $(\gamma_1(t), \gamma_2(t)F_2(x_1,\ldots,x_{n-1}))$, where
\begin{enumerate}
\item[(a)] $f(t)=t$, and $c=1+\tfrac{r}{(n-1)(n-2)}t^2$ is negative constant,

\item[(b)] $\gamma_1$ and $\gamma_2$ are determined by
\begin{equation*}
\gamma_1'=t^{n+1}\gamma_2^{-n},\
\gamma_2=k(t)^{1/(n+1)},
\end{equation*}
where $k(t)$ is a positive solution to the linear differential equation
\begin{equation*}
t^2k''(t)-(n+1)tk'(t)+(n+1)ck(t)=0.
\end{equation*}
\end{enumerate}

\item[(5)] The immersion
$(\gamma_1(t)x_1,\ldots,\gamma_1(t)x_{n-1},\tfrac12\gamma_1(t)\sum_{i=1}^{n-1}x_i^2
+\gamma_2(t),\gamma_1(t))$,
where
\begin{enumerate}
\item[(a)] $f(t)=t$, $c=1+\tfrac{r}{(n-1)(n-2)}t^2=0$,

\item[(b)] $\gamma_1$ and $\gamma_2$ are determined by
\begin{equation*}
\begin{aligned}
\gamma_1=(\tfrac{n+1}{n+2} t^{n+2}+c_1)^{\tfrac{1}{n+1}},\
\gamma_2'=\tfrac{n+1}{n+2}\gamma_1'\ln t-\tfrac{\gamma_1}{(n+2)t},
\end{aligned}
\end{equation*}
where $c_1$ is a constant.
\end{enumerate}

\item[(6)] The immersion $(x_1,\ldots,x_{n-1},
\tfrac12\sum_{i=1}^{n-1}x_i^2-\tfrac{1}{n+2}\ln t,\tfrac{1}{n+2}t^{n+2})$,
where $f(t)=t$, $c=1+\tfrac{r}{(n-1)(n-2)}t^2=0$.
\end{enumerate}
\end{theorem}
\begin{remark}\label{rem:4.3}
By Theorem \ref{thm:2.3}, the warped product structure in Theorem \ref{thm:4.1}
implies that all examples above, especially for $3$-dimensional ones, are conformally flat.
\end{remark}
\begin{proof}
Based on the previous analysis, we will discuss the three classes
of immersions in Theorem \ref{thm:2.2} by taking $m=n-1$, respectively.

\vskip 2mm
{\bf Case} I. It was shown in \cite{ADSV} that if
\begin{equation}\label{eqn:4.9}
\mu_2^2+(\alpha-\lambda_2)^2\neq0,\ \mu_2-\lambda_2^2+\alpha^2\neq 0,
\end{equation}
$M^n$ is affinely equivalent to the immersion
\begin{equation*}
(\gamma_1(t), \gamma_2(t)g_2(x_1,\ldots,x_{n-1})),
\end{equation*}
where $g_2:M_2\rightarrow\mathbb{R}^{n}$ is a proper affine hypersphere with
the difference tensor $L^2$ defined by \eqref{eqn:4.5}. It follows from the analysis above,
\eqref{eqn:4.7}-\eqref{eqn:4.9} that $L^2=0$, $g_2$ has constant sectional
curvature $c\neq0$, affine mean curvature $c$ and affine normal $-cg_2$.
Thus $g_2=F_1$ if $c>0$, or $g_2=F_2$ if $c<0$. As $r\neq0$,
\eqref{eqn:4.8} implies that $c\neq1$ if $f=t$.
Moreover, by the computations of this immersion on page 292-294 in \cite{ALVW}
we take $\lambda=-c$ in (4.3) of \cite{ALVW}, and deduce that
\begin{equation}\label{eqn:4.10}
\begin{aligned}
&\gamma_1'=f^{n+1}\gamma_2^{-n},\ \gamma_2=k(t)^{1/(n+1)},\\
&f^2k''(t)-(n+1)ff'k'(t)+(n+1)ck(t)=0,
\end{aligned}
\end{equation}
where $k(t)$ and thus $\gamma_2$ are positive functions.

If $f=1$, we see from \eqref{eqn:4.10} that $\gamma_1'=\gamma_2^{-n}$
and $\gamma_2=k(t)^{1/(n+1)}$, where
$$
k''(t)+(n+1)ck(t)=0.
$$
Solving this equation we obtain that
\begin{equation}\label{eqn:4.11}
\gamma_2=
\left\{
\begin{array}{lll}
(c_1e^{\sqrt{-(n+1)c}t}+c_2e^{-\sqrt{-(n+1)c}t})^{\tfrac{1}{n+1}}, &{\rm if}\  c<0,\\[2mm]
(c_1\cos(\sqrt{(n+1)c}t)+c_2\sin(\sqrt{(n+1)c}t))^{\tfrac{1}{n+1}}, &{\rm if}\  c>0,
\end{array}
\right.
\end{equation}
where the constants $c_1, c_2$ are chosen such that $\gamma_2>0$.
We obtain the immersion (1) if $c>0$, and the the immersion (2) if $c<0$.

If $f=t$, we see from \eqref{eqn:4.2}, \eqref{eqn:4.8} and
\eqref{eqn:4.10} that
$\gamma_1'\gamma_2^{n}=t^{n+1}$ and
\begin{equation}\label{eqn:4.12}
\begin{aligned}
&\gamma_2=k(t)^{1/(n+1)},\ c\neq1,\\
&t^2k''(t)-(n+1)tk'(t)+(n+1)ck(t)=0.
\end{aligned}
\end{equation}
In particular, if $k(t)$ is a power function of $t$,
we can solve this equation to obtain
\begin{equation}\label{eqn:4.13}
\gamma_2(t)=
\left\{
\begin{array}{lll}
c_1t^{\tfrac{n+2}{2(n+1)}},  &{\rm if} \ c=\tfrac{(n+2)^2}{4(n+1)},\\[2mm]
(c_2t^{\tau_1}+c_3t^{\tau_2})^{\tfrac{1}{n+1}},
&{\rm if} \ c<\tfrac{(n+2)^2}{4(n+1)},\\[2mm]
0, &{\rm if} \ c>\tfrac{(n+2)^2}{4(n+1)},
\end{array}
\right.
\end{equation}
where $c_1$ is a positive constant, the constants $c_2, c_3$ are chosen such that $\gamma_2>0$,
and $\tau_1, \tau_2$ are the solutions of the quadric equation
$\tau^2-(n+2)\tau+(n+1)c=0$. We obtain the immersion (3) if $c>0$,
and the immersion (4) if $c<0$.

\vskip 3mm
{\bf Case} II. It was shown in \cite{ADSV} that if
\begin{equation}\label{eqn:4.14}
\mu_2^2+(\alpha-\lambda_2)^2\neq0,\ \mu_2-\lambda_2^2+\alpha^2= 0,
\end{equation}
$M^n$ is affinely equivalent to the immersion
\begin{equation*}
\begin{aligned}
(\gamma_1(t)x_1,\ldots, \gamma_1(t)x_{n-1},
\gamma_1(t)g(x_1,\ldots,x_{n-1})+\gamma_2(t),\gamma_1(t)),
\end{aligned}
\end{equation*}
where $g(x_1,\ldots,x_{n-1})$ is a convex function whose
graph immersion is a parabolic affine hypersphere with the
difference tensor $L^2$ defined by \eqref{eqn:4.5}.
It follows from the analysis above, \eqref{eqn:4.7}, \eqref{eqn:4.8}
and \eqref{eqn:4.14} that $L^2=0$, $f=t$, $c=0$ and this parabolic affine hypersphere
has a flat affine metric, and thus $g(x_1,\ldots,x_{n-1})=\tfrac12(x_1^2+\cdots+x_{n-1}^2)$.

Moreover, we recall from the computations of such hypersurfaces on
page 294 of \cite{ALVW} that $\gamma_1,\gamma_2$ satisfy
\begin{equation}\label{eqn:4.15}
(\gamma_1'\gamma_2''-\gamma_1''\gamma_2')f^2=\gamma_1\gamma_1',\
f=|\gamma_1^n\gamma_1'|^{1/(n+1)}.
\end{equation}
As $f=t$, we have
$(\gamma_1^{n+1})'=\varepsilon(n+1)t^{n+1}$, $\varepsilon\in\{-1,1\}$,
which gives that $\gamma_1^{n+1}=\tfrac{n+1}{n+2}\varepsilon t^{n+2}+c_1$.
By applying an affine reflection we may assume that $\gamma_1>0$,
then put $\varepsilon=1$ and $\gamma_1=(\tfrac{n+1}{n+2} t^{n+2}+c_1)^{1/(n+1)}$.
By \eqref{eqn:4.15} we get
$$
(\gamma_2'/\gamma_1')'=t^{-2}\gamma_1/\gamma_1'
=\tfrac{n+1}{n+2}t^{-1}+ c_1t^{-n-3},
$$
which yields that
$$
\gamma_2'/\gamma_1'=\tfrac{n+1}{n+2}\ln t
-\tfrac{ c_1}{(n+2)t^{n+2}}+c_2.
$$
Then, since $\gamma_1^{n}\gamma_1'= t^{n+1}$
and $c_1=\gamma_1^{n+1}-\tfrac{n+1}{n+2} t^{n+2}$,
we have
$$
\gamma_2'=\tfrac{n+1}{n+2}\gamma_1'\ln t
-\tfrac{\gamma_1}{(n+2)t}+(\tfrac{n+1}{(n+2)^2}+c_2)\gamma_1',
$$
and
$$
\gamma_2=\int\big(\tfrac{n+1}{n+2}\gamma_1'\ln t
-\tfrac{\gamma_1}{(n+2)t}\big)dt
+(\tfrac{n+1}{(n+2)^2}+c_2)\gamma_1+c_3.
$$
Here, by applying equiaffine transformations
we may put $c_2=-(n+1)/(n+2)^2$ and $c_3=0$.
We have the immersion (5).

{\bf Case} III. It was shown in \cite{ADSV} that if
\begin{equation}\label{eqn:4.16}
\mu_2^2+(\alpha-\lambda_2)^2=0,
\end{equation}
$M^n$ is affinely equivalent to the immersion
\begin{equation*}
\begin{aligned}
(x_1,\ldots,x_{n-1}, g(x_1,\ldots,x_{n-1})+\gamma_1(t),\gamma_2(t)),
\end{aligned}
\end{equation*}
where $g(x_1,\ldots,x_{n-1})$ is a convex function whose graph immersion
is a parabolic affine hypersphere with the difference tensor $L^2$ defined by \eqref{eqn:4.5}.
It follows from the analysis above, $\lambda_2\neq0$ in \eqref{eqn:4.2},
\eqref{eqn:4.7}, \eqref{eqn:4.8} and \eqref{eqn:4.16} that $L^2=0$, $\mu_2=0$,
$\lambda_2=\alpha=-\tfrac{1}{t}$, $f=t$ and $c=0$.
As before $g(x_1,\ldots,x_{n-1})=\tfrac12(x_1^2+\cdots+x_{n-1}^2)$.

Moreover, we see from the computations of such hypersurfaces on
page 295 in \cite{ALVW} that $\gamma_1,\gamma_2$ satisfy
\begin{equation}\label{eqn:4.17}
\begin{aligned}
\gamma_2'^3=(\gamma_1''\gamma_2'-\gamma_1'\gamma_2'')f^{2(n+2)},\
f=\mid\gamma_2'\mid^{1/(n+1)}.
\end{aligned}
\end{equation}
Then, taking $f=t$ into this equation we deduce that
\begin{equation}\label{eqn:4.18}
\gamma_2'=\epsilon t^{n+1},\ t^2\gamma_1''-(n+1)t\gamma_1'-1=0,
\end{equation}
where $\epsilon\in\{-1,1\}$.
Then, we can directly solve these equations to obtain
$$
\gamma_2(t)=\tfrac{\epsilon}{n+2}t^{n+2}+c_1,\
\gamma_1(t)=-\tfrac{\ln t}{n+2}+c_2t^{n+2}+c_3.
$$
By applying a translation and a reflection in $\mathbb{R}^{n+1}$
we may assume that $c_1=c_3=0$, and $\gamma_2>0$, i.e., $\epsilon=1$. Also,
by possibly applying an equiaffine transformation
we may put $c_2=0$. Hence, we obtain the immersion (6).
\end{proof}

\vskip5mm

\end{document}